\documentclass[12pt]{amsart}

\usepackage[left=3cm,top=3.0cm,right=3cm,bottom=2.6cm]{geometry}

\usepackage[ansinew]{inputenc}
\usepackage{amscd,amsfonts,amsmath,amssymb,amsthm,enumerate,epsfig,float,graphics,graphicx,pst-all,yhmath}

\newcommand{\inte}{\operatorname*{int}}
\newcommand{\aff}{\operatorname*{aff}}
\newcommand{\bd}{\operatorname*{bd}}

\newcommand{\Ed}{\mathbb{E}^{d}}
\newcommand{\Sd}{\mathbb{S}^{d-1}}

\newtheorem{lemma}{Lemma}

\newtheorem{theorem}{Theorem}

\newtheorem{question}{Question}

\title [A characterization of the ellipsoid]{A characterization of the ellipsoid in terms of pairs of sections associated by a harmonic homology}
\author{Efr\'en Morales-Amaya}
\address{Facultad de Matem\'aticas-Acapulco,
Universidad Aut\'onoma de Guerrero, M\'exico}
\email{emoralesamaya@gmail.com}
\thanks{I would like to thank to Jes\'us Jer\'onimo-Castro for the interesting conversations about this work.
}   
\begin{document}

\maketitle
          \begin{abstract}  Let $K$ be  a convex body in an affine chart of the $n$ dimensional real Projective 
         space $\mathbb{RP}^n$, $n \geq 3$, let $H$ be a hyperplane 
          which is not a support hyperplane of $K$ and let $p_1,p_2 \in \mathbb{RP}^n \setminus H$ be two 
          distinct interior points of $K$. In this work we prove that if for
          for every $(n-2)$-plane $l \subset H$, there exists a harmonic 
          homology, with plane $G$ and center $\tau$, such that  $l\subset G$, $\tau \in H$ and  
          which maps the hypersection of $K$ defined by $\aff \{p_1, l\}$ 
          onto the hypersection of $K$ defined by $\aff \{p_2, l\}$, then 
           $K$ is an ellipsoid.
                \end{abstract}
          \section{Introduction}     
         A \textit{homology} is a collineation of the real projective plane $\mathbb{RP}^2$ on 
         itself, other than the identity, which leaves fixed every point of a line $G$ 
         and every line through a point $\tau$ not on $G$ (\cite{BusemannKelly}). 
         The line $G$ and the point $\tau$ are called respectively the 
       \textit{axis} and the \textit{center} of the homology. A homology is called 
       \textit{harmonic} it it has the property that a general pair of 
       correspondent points, $z$ and $z'$, are separated harmonically by the
       center $\tau$ and the point in which $L(z, z')$ intersects the axis $G$. In 
       natural way, this notion can be extended to any dimension.

         Let $K $ be a convex body in an affine chart of the $d$ dimensional real Projective 
         space $\mathbb{RP}^d$, let $H \subset \mathbb{RP}^d$ be a hyperplane which is not a 
         support hyperplane of $K$ and let $p_1,p_2 \in \mathbb{RP}^n \setminus H$ be two distinct interior 
         points of $K$. For every $(d-2)$-plane 
         $l\subset H$, we denote by $\Pi_i$ the hyperplane defined by $l$ and $p_i$ and let 
         $K_l ^i:=\Pi_i\cap K$, $i=1,2$.  
         
         Our main result is the following
         \begin{theorem}\label{Claudia}
         Let $K $ be a convex body in an affine chart of $\mathbb{RP}^d$, let 
         $H \subset \mathbb{RP}^d$ be a hyperplane and let 
         $p_1,p_2\in \inte K$, $p_i\notin H$, $i=1,2$. Suppose that for every $(d-2)$-plane
         $l\subset H$ there exists an harmonic homology 
         $\Phi_G^{\tau} :\mathbb{P}^d \rightarrow \mathbb{P}^d$ with hyperplane $G$ and 
         center $\tau$ such that $l\subset G$, $\tau \in H$ and 
         \[
         \Phi_G^{\tau}(K_l^1)=K_l^2.
         \]
         Then $K$ is an ellipsoid.
         \end{theorem}  
         
        The strategy for proving the Theorem \ref{Claudia} consists of three steps 
        (see subsection \ref{projective} for the corresponding definitions):
         \begin{enumerate}
         \item To prove that $K$ has a pole $g$ with polar $H$ (Lemma \ref{pole}),

          \item  $p_i$ is a false pole of $K$ with respect to the plane $H$ (Lemma \ref{falsepole}) and

         \item to use the False Pole Theorem (see Sec 2.2).
         \end{enumerate}
  
         \section{Preliminaries and basic notions}

       Let $\mathbb{E}^{d}$ be the Euclidean space of dimension $d$ endowed with the usual scalar 
       product $\langle \cdot, \cdot\rangle : \mathbb{E}^{d} \times \mathbb{E}^{d} \rightarrow \mathbb{R}$. 
       Let $\mathbb{S}^{d-1}=\{x\in \mathbb{E}^{d}: \|x\| = 1\}$ be the unit sphere in $\mathbb{E}^{d}$, for $v \in \Sd$ we denote by $v^{\perp}$ 
       the hyperplane through the origin perpendicular to $v$. Let $x, y \in \mathbb{R}^n$, we denote by $L(x, y)$ the line through $x$ and $y$, 
       and by $[x, y]$ the line segment connecting them. If $W\subset \mathbb{E}^{d}$, we denote by $\aff W$ the affine hull of $W$.

        A {\it convex body} $K\subset {\mathbb E^d}$, $d\ge 2$,   is a convex compact set with non-empty interior. 
        A \textit{convex hypersurface} is the boundary of a convex body $K$ in $\mathbb{E}^{d}$ and it will be 
        denoted by $\bd K$. We will denote by $\inte K$ the set $K\backslash \bd K$.
        A {\it chord} of a convex body $K$ is any line segment $[x, y]$ in $K$ 
	such that $x,y\in \bd K$. An excellent book 
        where you can consult the basic concepts and results of convexity is \cite{Constantwidth}.
                 
        Let $W\subset \mathbb{E}^d $ be a compact convex set. Given a point 
         $x \in \mathbb{E}^d \backslash \aff \{W\}$, we denote by $S_x(W)$ the \textit{cone generated by $W$ 
         with apex $x$}, that is, 
         $S_x(W) := \{x + \lambda(y - x) : y \in W, \lambda \geq  0\}$ 
         and by $C_x(W)$ the boundary of $S_x(W)$.

         \subsection{Some affine notions and results} 
         A body $K\subset \mathbb{E}^d$ is origin symmetric if when ever $x \in K$, it follows that 
         $-x \in K$. A body K is \textit{centrally symmetric} or has a \textit{center} 
         if a translate of $K$ is origin symmetric, i.e., if there is a vector 
         $c\in \mathbb{E}^d$ such that $K-c$ is origin symmetric.
         
         Let $G \subset \mathbb{E}^d$ be a hyperplane and let $u$ be a unit vector 
         not parallel to $G$. A mapping $S_G^u : \mathbb{E}^d \rightarrow \mathbb{E}^d$ is 
         an \textit{affine reflection} with respect to $G$ if, for every point 
         $x \in \mathbb{E}^d$, the point $S_G^u(x)$ lies on the line parallel to $u$ through 
         $x$, at equal distance from $G$, and on the opposite side of $G$ from $x$; 
         $u$ and $G$ are called the \textit{direction} and the \textit{hyperplane of the affine
         symmetry}. A convex body $K \subset \mathbb{E}^d$ is said to be 
         \textit{affine symmetric with respect to $S_G^u$} if $S_G^u(K )= K$. In particular, if $u$ is orthogonal 
         to $G$, we say that the body $K$ is symmetric.  
            
         \textbf{False Center of Symmetry of a convex body.} We say that the point 
         $p\in \mathbb{E}^d$ is a \textit{false center} of $K$  if: (1) $p$ is not a center 
         of $K$ and (2) for every 2-plane $\Pi$ passing through $p$, the section 
         $\Pi \cap K$ is centrally symmetric (in the British language the word center 
         is \textit{centre} but in this work we will adopt the American word).
         
         In \cite{ro1} it is proved that a convex body $K\subset \mathbb{E}^d$, $d\geq 3$, with an 
         interior false center is centrally symmetric, and there it was conjectured that a convex body with a false 
       center is an ellipsoid. This problem was solved in \cite{ai-pe-ro} for the case when the false 
       center is in the interior of $K$ and in \cite{la} was presented the solution for the other cases. Both 
       proofs are quite long and complicated. In \cite{mm2} was given a short proof of the False Center 
       Theorem (FCT) and in \cite{mm3} a generalization of the FCT was proposed.

         \subsection{Some projective notions and results} \label{projective}
         
         Let $\mathbb{E}^d$ be the $d$-dimensional Euclidian space. We complete 
        $\mathbb{E}^d$ to the $d$-dimensional real projective space $\mathbb{RP}^d$ by 
        adding the hyperplane at the infinity $H_{\infty}$.
         Two points $P_1,P_z$ on a line, for which the division ratio with respect to $A,B$ is equal in 
         absolute value but different in sign, are called \textit{harmonic with respect to $A$ and $B$} 
         \cite{Struik}. $P_1$ the \textit{harmonic conjugate} of $P_2$, and conversely. We also say simply 
         that the four points are \textit{harmonic} and it will be denoted as 
         \[
         [A,B;P_1,P_2=-1].
         \] 
         The definition requires that for such points the relation
         \[
         P_1A:P_1B=-P_2A:P_2B
         \]
         holds.  
           
       The \textit{complete quadrangle} is formed by four points $A,B,C,D,$ connected by the six sides 
       $L(A,B),L(A,C), ...,L(C,D)$. These sides can be separated into three pairs of opposite sides, 
       which intersect in the three diagonal points $P,Q,R$: 
       \[
       L(A,B), L(C,D) \textrm{ }\textrm{ in }P;\textrm{ }\textrm{ }L(A,C),L(D,B)\textrm{ }\textrm{ in }Q;
       \textrm{ }\textrm{ }L(A,D),L(B,C) \textrm{ }\textrm{ in }R.
       \]  
        The following theorem will be used in the proof of the Theorem \ref{Claudia} (see pag. 46 of \cite{Struik}):
        
        \textbf{[Complete quadrangle Theorem]} 
          \textit{On every side of a complete quadrangle 
         the two vertices are harmonic with respect to the points in which this side is intersected by the sides 
         of the diagonal triangle.}
         
         We repeat the dual version of Pappus' theorem (see Sec. 2-7 and pag. 44 of \cite{Struik}):
                 
                 \textbf{[Dual version of Pappus' theorem]} 
         \textit{Two lines are intersected by four lines of a pencil in points with same cross ratio}.

         Let $K$ be a convex body in 
        $\mathbb{E}^d$, $d\geq  2$, i.e., a compact, convex set with non-empty interior. 
        The point $p\in \mathbb{RP}^d$ is said to be a \textit{pole} of $K$ (or a 
        \textit{projective center of symmetry} of K) with respect to the hyperplane 
        $H\subset \mathbb{RP}^d$ if for every line $L$ passing through $p$, we have 
       \[ 
       [A, B; p, q  ] =- 1,
       \]
        where $\{A,B\}:= L \cap K$, $q:=L \cap H$. 
        Under these circumstances, the hyperplane $H$ is called the polar (or a 
        projective hyperplane of symmetry) of the pole $p$. Some properties of this 
        notion can be found in \cite{mm}.
       
       Notice that if the point $p\in \mathbb{RP}^d$ is a pole of $K$ with polar $H$ 
       and we define the harmonic homology 
       $\Phi_H^{p}:\mathbb{RP}^d \rightarrow \mathbb{RP}^d $ 
       with center $p$ and hyperplane $H$, then $\Phi_H^{p}(K)=K$. 
        
       \textbf{The False Pole Theorem.} We say that the point $p\in \mathbb{RP}^d$ is 
       a \textit{false pole} of $K$ with respect to the hyperplane $H$ if: (1) $p$ is not a 
       pole of $K$ with respect to $H$ and (2) for every 2-plane $\Pi$ passing 
       through $p$, the section $\Pi \cap K$ has a pole whose corresponding polar is 
       $\Pi \cap H$.
       
       Notice that if $E\subset \mathbb{RP}^d$, $d\geq 3$, is an ellipsoid, 
       $p\in \mathbb{RP}^d$ is a pole of $E$ with polar $Q_p$ and $H$ is a hyperplane 
       such that $H\not=Q_p$, then $p$ is a false pole of 
       $K$ with respect to $H$. The next result was proved in \cite{LarmanMora}:
       
       \textbf{[False Pole Theorem].}
       \textit{Let $K\subset \mathbb{E}^d$, $d\geq 3$ be a convex body. If $K$ has an 
       interior false pole $p\in \mathbb{E}^d$ of $K$ with respect to the hyperplane 
       $H$, then $K$ is an ellipsoid.}
       
       We observe that if $p$ is a false pole of $K$ with respect to $H_{\infty}$, then 
       $p$ is a \textit{false centre of symmetry} of $K$, because $K$ is not centrally 
       symmetric at $p$ and every 2-section of $K$ passing through $p$ is centrally 
       symmetric.
       
       \section{Motivation}
      Let us consider the following question. 
       \begin{question}\label{motivation1}
       Let $K  \subset  \Ed$, $d\geq 3$, be a convex body and let $p_1, p_2\in \inte K$, $p_1\not=p_2$. 
       Su\-ppo\-se that, for every  $v \in \Sd$, there exists an affine reflection 
       $S_{v^\perp}^u: \Ed \rightarrow \Ed$, for some direction $u$ and hyperplane of 
       reflection $v^{\perp}$, such that 
       \begin{eqnarray}\label{danzonero}
       S_{v^\perp}^u \big{(} [p_1+v^{\perp}] \cap K \big{)}= [p_2+ v^{\perp}] \cap K.
        \end{eqnarray}
        What can we say about $K$?
        \end{question}
        
        One of the main motivations of this work is to study Question \ref{motivation1}, but not in the context of Affine Geometry, 
        rather in that of Projective Geometry. Later, we will answer Question \ref{motivation1} in the context of Affine Geometry and 
        see some reasons why that answer doesn't work when we move from one geometry to another, and from there, we will deduce 
        the need to provide an answer in the projective context.
         
         Next, we will show how the Question \ref{motivation1} arose. The following result is due to Rogers  \cite{ro1} (In fact, Rogers' result is 
        more general but for our purposes the next formulation is more relevant):
        
                         \textbf{[Rogers' Theorem]} 
       \textit{Let $K_1, K_2 \subset \mathbb{E}^d$, $d \geq 3$, be two convex bodies and let $p_i \in \inte K_i$, $i=1,2$.  
          Su\-ppo\-se that, for every  $v \in \Sd$, there exists a translation 
       $T_ u: \Ed \rightarrow \Ed$, for some direction $u$, such that 
       \begin{eqnarray}\label{rockero}
       T_u \big{(} [p_1+v^{\perp}] \cap K_1 \big{)}= [p_2+ v^{\perp}] \cap K_2.
        \end{eqnarray}
        Then there exists a translation $T: \Ed \rightarrow \Ed$ such that $T(K_1)=K_2$.}
        
        There are two possibilities: 1) $T(p_1)=p_2$ and 2) $T(p_1)\not=p_2$. Let us focus on case 2).
        In such case we have the following:       
        \begin{question}\label{motivation2}
       Let $K  \subset  \Ed$, $d\geq 3$, be a convex body and let $p_1, p_2\in \inte K$, $p_1\not=p_2$. 
       Su\-ppo\-se that, for every  $v \in \Sd$, there exists a translation 
       $T_u: \Ed \rightarrow \Ed$, for some direction $u$,  such that 
       \begin{eqnarray}\label{danzonero}
       T_u \big{(} [p_1+v^{\perp}] \cap K \big{)}= [p_2+ v^{\perp}] \cap K.
        \end{eqnarray}
        What can we say about $K$?
        \end{question}  
        Since the translation for a vector $u$ between two parallel sections of $K$ passing through $p_1,p_2$  can be interpreted as an 
        affine reflection, with respect to a hyperplane parallel to the section and with direction $u$, we can see that now we are on the conditions 
        of Question \ref{motivation1}, that is, The Question \ref{motivation1} and Question \ref{motivation2} are equivalent.  
         
         On the other hand, the next result, in the spirit of Rogers' theorem, has a corollary (mentioned as a particular case) related to Question \ref{motivation1}, 
         notice that in this result were considered only orthogonal reflections, it was proved in \cite{Efren1}:
 
          \textit{If for a pair of convex bodies $K_1, K_2 \subset \mathbb{E}^d$, $d \geq 3$, there exists a hyperplane 
         $H$ and two distinct points $p_1,p_2 \in \mathbb{E}^d \setminus H$ such that for every $(d-2)$-plane 
         $l \subset H$, there exists a reflection mapping the hypersection of $K_1$ defined by 
         $\mathrm{aff}\{p_1, l\}$ onto the hypersection of $K_2$ defined by $\mathrm{aff}\{p_2, l\}$, then the 
         reflection with respect to $H$ maps $K_1$ onto $K_2$. }

          \textit{In particular, if $K_1=K_2$, then $K_1$ has $H$ as hyperplane of 
         symmetry.}
 	 
	  \textbf{Answer to Question \ref{motivation1}.} 
        By Larman-Tamvakis' Theorem \cite{larmantamvakis}, $K$ is centrally symmetric with centre at the 
        mid-point of $[p_1,p_2]$. Thus, if we take a system of coordinates with the origin at the mid-point 
        of the line segment $[q_1,q_2]$, for all $v \in \Sd$, the relation
         \begin{eqnarray}\label{elizabeth}
        -([p_1+v^{\perp}]\cap K)=[p_2+v^{\perp}]\cap K, 
         \end{eqnarray}
         holds. By (\ref{danzonero}) and (\ref{elizabeth}) it follows that
         \begin{eqnarray}\label{papeleria}
        -([p_1+v^{\perp}]\cap K)=S_{v^\perp}^u \big{(} [p_1+v^{\perp}] \cap K \big{)}.
         \end{eqnarray}
         The relation (\ref{papeleria}) implies that $[p_1+u^{\perp}]\cap K$ 
         is centrally symmetric, i.e., all the hypersections of $K$ passing through $p_1$ are centrally 
         symmetric. Since $p_1$ is not the center of $K$, $p_1$ is a false center of $K$ and, 
         consequently, by virtue of the FCT, $K$ is an ellipsoid. $\square$
         
         It is possible to interpret an affine reflection $S_G^u : \mathbb{E}^d \rightarrow \mathbb{E}^d$  
          regarding the hyperplane $G$ and the direction $u$ as a harmonic homology 
          $\Phi_G^{\tau} :\mathbb{P}^d \rightarrow \mathbb{P}^d$ with hyperplane $G$ and 
         center $\tau$, where $\tau$ is the point in the hyperplane at infinite $H_{\infty}$ defined by $u$.
    
          Furthermore, it is possible to interpret that the parallel sections of $K$ passing through points $p_1$, 
          $p_2$ are defined by hyperplanes that intersect at $H_{\infty}$. Thus, we see that Question 
          \ref{motivation1}, which is 
          formulated in the affine context, can be extended naturally to Projective Geometry. This allows the 
          hyperplane $H$ to be any hyperplane, which is not a supporting hyperplane of $K$, and not only the 
          hyperplane $H_{\infty}$.
           
           However, as soon as we do this, we realize that we no longer have a projective version of the 
           Larman-Tamvakis theorem, which was the starting point of our answer to Question \ref{motivation1}. 
           That is why an original answer to the general question is necessary. The Theorem \ref{Claudia} does 
           precisely that: it establishes the statement that defines the general situation and provides a proof.
       \section{Proof of Theorem \ref{Claudia}}
          
         Let $h:=L(p_1,p_2)\cap H$ and let 
         $g\in  L(p_1,p_2)$ such that 
                \begin{eqnarray}\label{razoncruz}
                [p_1,p_2;g,h]=-1.
                \end{eqnarray}
               \begin{lemma}\label{pole}
              The point $g$ is a pole of $K$ with polar $H$.
              \end{lemma}
              \begin{proof}
              
              Let $r_1,r_2\in \bd K$ such that $g\in L(r_1,r_2)$ and let $k:=L(r_1,r_2)\cap H$. 
              We are going to prove that $[r_1,r_2;g,k]=-1$.
              
              Let $\Pi_1$ be a hyperplane containing $L(p_1,r_1)$ but not containing 
              $L(r_1,r_2)$, let $l:=\Pi_1\cap H$ and let $\Pi_2$ be the hyperplane defined by 
              $l$ and $p_2$. By virtue of the hypothesis there exists an harmonic 
              homology 
               \[
              \Phi_G^{\tau} :\mathbb{P}^d \rightarrow \mathbb{P}^d \textrm{ } \textrm{ with }\textrm{ } \textrm{ plane } \textrm{ }G 
              \textrm{ } \textrm{ and } \textrm{center} \textrm{ }
              \tau 
              \]
              such that 
              \[
              l\subset G, \tau \in H\textrm{ } \textrm{ and }\textrm{ } \Phi_G^{\tau} (K_l^1)=K_l^2,
              \]
              where $K_l^i:=K\cap \Pi_i$, $i=1,2$. Notice that, in particular, this yields
              \[
              \Phi_G^{\tau}(\Pi_1)=\Pi_2.
              \]
              Let $q_i:=L(g, \tau)\cap \Pi_i$, $i=1,2$ and let $\Gamma$ be the 
              2-plane defined by the lines $L(p_1,p_2)$, $L(q_1,q_2)$. Then the lines 
              \[
              \Gamma \cap \Pi_1, \Gamma \cap \Pi_2, \Gamma \cap G \textrm{ } \textrm{ and }\textrm{ }\Gamma \cap H
              \] 
              defines a harmonic lines of a pencil and the lines $L(p_1,p_2)$ and $L(q_1,q_2)$ are two transversal of it (notice that since 
              $\Pi_1$ is not containing $L(r_1,r_2)$) the point $g$ does not belong to $\Pi_1$ and, consequently, 
              $L(p_1,p_2), L(q_1,q_2)$ are not contained in $ \Pi_1$ and $\Gamma$ is not contained in $\Pi_1$). Thus,
              by the dual version of Pappus' theorem:
              \[
               [p_1,p_2;g',h]= -1=[q_1,q_2;g'',\tau],
               \] 
               where $g':=L(p_1,p_2) \cap G$, $g'':=L(q_1,q_2) \cap G$. Hence 
               $g'$ is harmonic conjugate of $h$ with respect to $p_1,p_2$. 
               On the other hand, by (\ref{razoncruz}) $g$ is harmonic conjugate of $h$ with respect to 
               $p_1,p_2$. Thus, by virtue of 
               the uniqueness of the conjugate harmonic of $h$ with respect to $p_1,p_2$, it follows 
               that $g'=g=g''$ and $g\in G$.
              
              On the other hand, let $\Delta$ be the plane defined by the lines $L(p_1,p_2)$, $L(r_1,r_2)$. 
              Then the lines $\Delta \cap \Pi_1$, $ \Delta \cap \Pi_2$, $ \Delta \cap G$ and $ \Delta \cap H$ 
              defines a harmonic lines of a pencil and the lines $L(p_1,p_2)$ and $L(r_1,r_2)$ are two 
              transversal of it. Thus, by the dual version of Pappus' theorem, $ [p_1,p_2;g,h]= [r_1,r_2;g,k]$ 
              and, by (\ref{razoncruz}), $[r_1,r_2;g,k ]=-1$.
              By virtue of the arbitrariness of $r_1,r_2\in \bd K$ such that $g\in L(r_1,r_2)$, we conclude that 
              $g$ is pole of $K$ with polar $H$.
              \end{proof}
              
               \begin{lemma}\label{falsepole}
               The point $p_i$ is a false pole of $K$ with respect to the hyperplane $H$, $i=1,2$.
               \end{lemma}
               
              \begin{proof}
              In order to prove that the point $p_1$ is a false pole of $K$ with respect to 
              the hyperplane $H$, we must show that, for every hyperplane $\Pi_1$ passing through 
              $p_1$ and not containing $L(p_1,p_2)$, there exists a pole $q_1$ of the 
               section $K_l^1:=K\cap \Pi_1$ with polar $l$, where $l:=\Pi_1\cap H$. 
               Notice that, by Lemma \ref{pole}, $g$ is a pole of 
               $K$ with polar $H$ and, since $\Pi_1$ is not containing $L(p_1,p_2)$, it 
               follows that $q_1\not=g$, that is, $p_1$ is a false pole of $K$ with respect to 
              the hyperplane $H$. 
              
              Let $\Pi_1$ be a hyperplane passing through $p_1$ and not containing 
              $L(p_1,p_2)$ and let $l:=\Pi_1\cap H$ and   
              $\Pi_2:= \aff\{p_2, l\}$. By the hypothesis, there exists 
              an harmonic homology 
              \[
              \Phi_G^{\tau} :\mathbb{P}^d \rightarrow \mathbb{P}^d \textrm{ } \textrm{ with }\textrm{ } \textrm{ plane } \textrm{ }G 
              \textrm{ } \textrm{ and } \textrm{center} \textrm{ }
              \tau 
              \]
              such that 
              \[
              l\subset G, \tau \in H\textrm{ } \textrm{ and }\textrm{ } \Phi_G^{\tau} (K_l^1)=K_l^2,
              \]
              where $K_l^i:=K\cap \Pi_i$, $i=1,2$. Let 
              $q_i:=L(g, \tau)\cap \Pi_i$, $i=1,2$. Notice that
              \begin{eqnarray}\label{penameda}
               [q_1,q_2;g,\tau]=-1.
              \end{eqnarray}
              Let $\Gamma$ be a 2-plane containing 
              the line $L(q_1,q_2)$. Let 
              \[
              \{T_1,T_2\}:=\Gamma \cap C_{\tau}(K_l^1), 
               T_3:=L(q_1,q_2), T_4:=\Gamma \cap H
               \] (we recall that $C_{\tau}(K_l^1)$ is the cone define by the section $K_l^1$ and apex $\tau$) and let 
               \[
               s_{ij}:=T_i\cap \Pi_j, 
              i,j=1,2. \textrm{ } \textrm{  and }\textrm{ }g':=L(s_{11},s_{22})\cap L(s_{12},s_{21})\textrm{ } (\textrm{ see } \textrm{ fig. } \ref{ninera}).
              \]     
               By virtue that $\Phi_G^{\tau}(\Gamma \cap \Pi_1)=\Gamma \cap \Pi_2$ it follows that 
              \[
              \Phi_G^{\tau}(s_{11})=s_{12} \textrm{ }\textrm{ and } \textrm{ }\Phi_G^{\tau}(s_{21})=s_{22}.
              \]               
      By Lemma \ref{pole}, $g$ is a pole of $K$ with polar $H$. Thus if we define the harmonic homology 
               \[
               \Phi_H^{p}: \Ed \rightarrow \Ed \textrm{ } \textrm{ } \textrm{ with }\textrm{ } \textrm{ hyperplane }\textrm{ } H\textrm{ } \textrm{ and }\textrm{ }\textrm{  center } p,
               \]
               \begin{figure}[H]
    \centering
     \includegraphics [width=.8\textwidth]{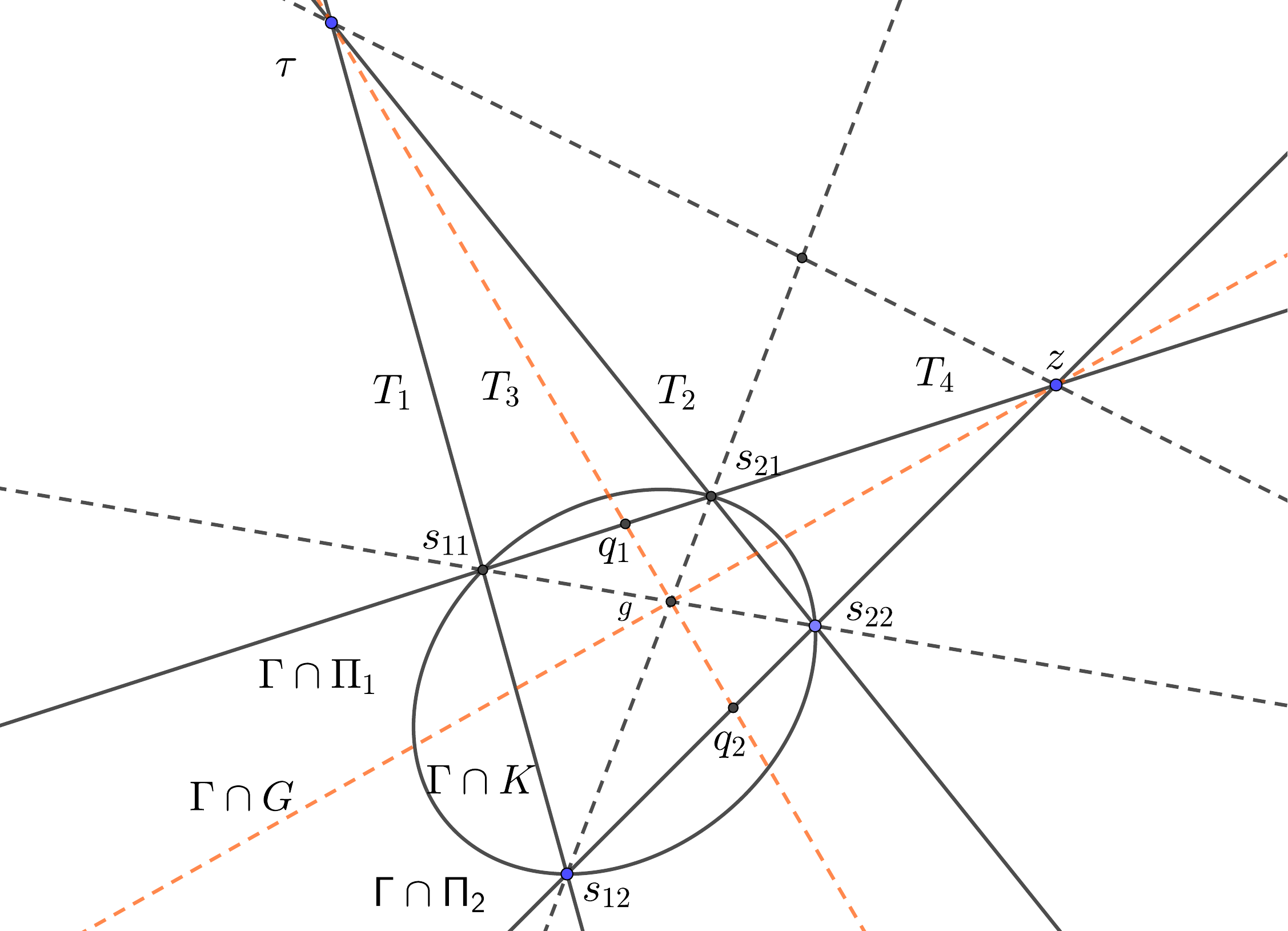}
    \caption{$p_1$ is a false pole of $K$ with respect to the hyperplane $H$, case $H\cap K=\emptyset$. In the case $H\cap K\not=\emptyset$, the drawing is the same; we only need to swap the line $\Gamma \cap H$ to $\Gamma \cap G$ and the point $\tau$ to $g$. In that case, we would have $g\notin \Gamma \cap K$ and $\tau \in \Gamma \cap K$.}
    \label{ninera}
     \end{figure}
                then $\Phi_H^{p}(K)=K$. By (\ref{razoncruz}) and (\ref{penameda}) and since $l\subset \Pi_1,\Pi_2$
                it follows that $\Phi_H^{p}(\Pi_1)= \Pi_2$, in particular,
               \[
                \Phi_H^{p}(\Gamma \cap \Pi_1)=\Gamma \cap \Pi_2.
               \]
                Consequently,      
                \[
                \Phi_H^{p}(s_{11})= s_{22}\textrm{ } \textrm{  and  }\textrm{ } 
               \Phi_H^{p}(s_{21})= s_{12}.
               \]               
              This means that $g'=g$ and the points
              \[
               \{s_{11}, s_{22}, s_{21}, s_{12}\}
               \]
               define a complete quadrangle whose pair of opposite sides $L(s_{11},s_{22})$, 
              $L(s_{21},s_{12})$ intersects at the diagonal point $g$. By the Complete Quadrangle Theorem  
              \[
              [s_{11},s_{21};q_1,z]=-1,
              \]
              where $z:=L(s_{11},s_{21})\cap H$. By virtue of the 
              arbitrariness of the 2-plane $\Gamma$ it follows that $q_1$ is a pole of $K_l^1$ with polar $l$. 
              
              The proof that the point $p_2$ is a false pole of $K$ with respect to 
              the hyperplane $H$ is analogous.
              \end{proof}
                         
            \textbf{Author Contributions.} Material preparation were performed by Efr\'en Morales 
           Amaya. The first draft of the manuscript was written by Efr\'en Morales Amaya.  
           
          \textbf{Funding.} The authors declare that no funds, grants, or other support were received 
          during the preparation of this manuscript.

          \textbf{Data Availability.} Data sharing not applicable to this article as no datasets were 
          generated or analyzed during  the current study.
          
          \textbf{Declarations.}
          
           \textbf{Conflict of interest.} The authors have no relevant financial or non-financial interests to disclose.


\begin{thebibliography}{99} 

\bibitem{ai-pe-ro}  Aitchison, P.W., Petty, C.M. and Rogers, C.A. A convex body with a false centre is an ellipsoid,  \textit{Mathematika} \textbf{18} (1971), $50-59$.
 

\bibitem{BusemannKelly} Busemann, H., Kelly, P. Projective Geometry and Projective Metrics. Academic Press Inc. New York 1953.


\bibitem{la} Larman, D.G. A note on the false center problem, \textit{Mathematika} \textbf{21} (1974), $216-27$.

\bibitem{larmantamvakis} D. Larman, Tamvakis.  A characterisation of centrally symmetric 
convex bodies in $\mathbb{E}^n$, \textit{Geom. Dedicata} \textbf{10} (1981) $161-176$.


\bibitem{LarmanMora} Larman, D., Morales-Amaya, E.  On the false pole 
problem. \emph{Mh Math} \textbf{151}, 271-286 (2007). https://doi.org/10.1007/s00605-007-0448-6

    \bibitem{DavidLuisEfren} Larman, D., Montejano, L. and Morales-Amaya, E. (2010), Characterization of the ellipsoid by means of parallel translated sections. \textit{Mathematika}, \textbf{56}: $363-378$. https://doi.org/10.1112/S0025579310000379

\bibitem{Constantwidth}  Martini, H., Montejano, L., Oliveros, D. \emph{Bodies of Constant Width: An Introduction to Convex Geometry with Applications}. Springer, 2019.


\bibitem{mm} Montejano, L., Morales, E.: Polarity in convex bodies: characterizations of ellipsoids. \textit{Mathematika} \textbf{50},
 (2003) $63-72$. 

\bibitem{mm2} Montejano, L and Morales-Amaya, E.: Variations of Classic Characterizations of Ellipsoids and a Short Proof of the False Centre Theorem. \textit{Mathematika} \textbf{54} (2007), $35-40$.

\bibitem{mm3} Montejano, L. and Morales-Amaya, E.: A shaken False Centre Theorem.  \textit{Ma\-the\-ma\-ti\-ka} \textbf{54} (2007), $41-46$.

\bibitem{Efren1} Morales-Amaya, E. Convex bodies with pairs of sections associated by reflections. \textit{Beitr Algebra Geom} (2025). https://doi.org/10.1007/s13366-025-00806-w
 
\bibitem{ro1} Rogers, C.A. : Sections and projections of convex bodies, \textit{Portugaliae Math.} \textbf{24} (1965), $99-103$.

\bibitem{ro2} Rogers, C.A. : An equichordal problem, \textit{Geom. Dedicata} \textbf{10} (1981),$73-78$.


\bibitem{Struik} Struik, D. J., Lectures on Analytic and Projective Geometry. Addison-Wesley Publishing Company, 1953.

\end{thebibliography}
\end{document}